\documentclass[12pt]{article}

\RequirePackage[OT1]{fontenc}
\usepackage{amsthm,amsmath}
\RequirePackage[numbers]{natbib}

\RequirePackage[colorlinks,citecolor=blue,urlcolor=blue]{hyperref}   

\def\MR#1{\href{http://www.ams.org/mathscinet-getitem?mr=#1}{MR#1}}     

\usepackage[utf8]{inputenc} 
\usepackage[english]{babel}

\usepackage{dblfloatfix}    
\usepackage{kantlipsum}     

\usepackage{graphicx, enumitem}

\usepackage[utf8]{inputenc}
\usepackage[english]{babel}
\usepackage{amssymb, amsfonts, verbatim, amstext, mathrsfs}
\usepackage[mathcal]{euscript}
\usepackage{bm}
\usepackage{bbm}
\usepackage{graphicx}
\usepackage{xcolor}
\usepackage{parskip}
\usepackage{multicol}
\usepackage{float}
\usepackage{rotating}

\newtheorem{theorem}{Theorem}[section]

\newtheorem{remark}[theorem]{Remark}

\newtheorem{definition}[theorem]{Definition}

\begin{document}

\title{ Local $L^{p}$-solution for semilinear  heat equation with  fractional   noise.}

\author{Jorge Clarke\thanks{CEREMADE, UMR CNRS 7534 Université Paris-Dauphine, PSL Research university,
Place du Maréchal de Lattre de Tassigny 75016 Paris, FRANCE. \newline
e-mail: \textsl{clarkemove@gmail.com.}}, 
Christian Olivera\thanks{Departamento de Matem\'{a}tica, IMECC, Universidade Estadual de Campinas, Brazil. 
\newline
e-mail: \textsl{colivera@ime.unicamp.br}. }}

\date{}

\maketitle


\vspace{0.3cm} \noindent {\bf MSC2010 subject classification:} 60H15, 60H30, 35R60, 35K05, 35K10
35K58. 

%
\begin{abstract}
We study the $L^{p}$-solutions for the  semilinear heat equation with unbounded coefficients and driven by a infinite dimensional fractional Brownian motion with self-similarity parameter $H > 1/2$. Existence and uniqueness of local mild solutions are showed. 
\end{abstract}

%
\maketitle

%

\section {Introduction} \label{Intro}

\quad The fractional Brownian motion, referred to as fBm in the sequel, due to its desirable properties of self-similarity and long-range dependence (among other features), has become quite a relevant stochastic process for mathematical modeling in engineering, mathematical finances, and natural sciences, to mention just a few. It was first introduced by Kolmogorov in \cite{Kolm}, and later, the work of Mandelbrot and Van-Ness \cite{Mandel-1968} became a corner-stone that attracted the attention of researchers in the probabilistic community to this challenging object.

\quad Nowadays, the study of ordinary and partial stochastic differential equations driven by a fractional noise is a very dynamic research topic, motivated by purely theoretical reasons and also by its variety of applications in the mathematical modeling of phenomena in physics, biology, hydrology, and other sciences. Besides, a special interest in the study of the existence and uniqueness of solutions to semilinear parabolic stochastic differential equations driven by an infinite-dimensional fractional noise has been recently developed (see for instance, T. E. Duncan, B. Pasik-Duncan and  B. Maslowski \cite{Dun}; D. Nualart and P. A. Vuillermot \cite{N2}; B. Maslowski and B.Schmalfuss \cite{Mas}, and M. Sanz-Sole and  P. A. Vuillermot \cite{Sans}, and the references therein).

\quad Other kind of driving noises have been also considered. In \cite{Bre},  Z. Brzezniak,  J. Neerven, D. Salopek, studied evolution equations with Liouville fractional Brownian motion; equations driven by Hermite or Rosenblatt process were adressed by S.  Bonaccorsi and C. Tudor in \cite{Bona}, and C. Tudor in \cite{Tudor}. More recently, equations driven by Volterra noises were analysed  by  P.  Coupek, B. Maslowski in  \cite{Co2} and by P.  Coupek, B. Maslowski, and M. Ondrejat in \cite{Co}. 

\quad In difference with the present manuscript, all the mentioned works consider a non-linearity $F$ that is bounded or with linear growth.

\quad Analogously to deterministic partial differential equations, the first obstacle is the requirement of deciding which kind of solution concept will be considered, due to the variety of alternatives to choose. We address the study of existence and uniqueness of mild-solution to the initial value problem for the semilinear heat equation over a smoothly bounded open domain $U \subset \mathbb{R}^{d}$, 

\begin{equation}\label{para}
 \left \{
\begin{aligned}
    & \partial_t u( t ) = \Delta u( t ) + F( u(t) ) +  \partial_{t} B^{H} (t), \quad t \in [0 , T] \\
    & u|_{t=0}=  u_{0}, 
\end{aligned}
\right .
\end{equation}
\quad In (\ref{para}),  $F$  represents the nonlinear part of the equation, $u_{0} \in L^{p}(U)$, and the random forcing field  $B^{H}$ is  a Hilbert space-valued  fractional Brownian motion defined on some complete probability space $(\Omega, \mathcal{F}, \mathbb{P})$.

\quad In this manuscript, the existence and uniqueness of local $L^{p}$-solutions for the stochastic parabolic equation (\ref{para}) with unbounded parameter $F$ and  $B^{H}$ a cylindrical fractional Brownian motion with selfsimilarity parameter $H > 1/2$, is proved. The approach to study $L^{p}$-solutions  is based in the concept of mild solution, which can be obtained by rewriting (\ref{para}) as an integral equation,

\[
u(t)= S(t)u_{0}+  \ \int_{0}^{t} S(t-s) F ( u(s) ) \,  ds +  \int_{0}^{t} S(t-s) dB^{H}(s), 
\]

and then  proving that,  in a suitable  function  space,  the  right-hand side  defines  a contraction. 

\quad Results on the existence of mild solutions with values in $L^{p}$ were established by  T. Giga  in \cite{Giga2},
A. Mazzucato in \cite{Mazu}, and F.B. Weissler  in \cite{Weis} and  \cite{Weis2} for the deterministic setting. 

\quad The rest of the manuscript is fashioned as follows. In Section \ref{Pre} the basic concepts, hypothesis and tools are introduced. The results are presented in Section \ref{Results}. The Appendix, Section \ref{Appendix}, contains the basic elements on how to extend the results to a more general setting.


\section{Preliminaries} \label{Pre}

\quad Hypothesis, background and some useful notation are introduced in what follows. Let $(\Omega, \mathcal{F}, \mathbb{P})$ be a complete probability space.

\subsection{Fractional Brownian motion}

\quad Let $T > 0$ be a fixed time horizon. Recall that a one-dimensional fractional Brownian motion $(b^{H}(t))_{t \in [0 , T]}$ with Hurst  parameter $H\in(0,1)$, is a centred Gaussian  process with covariance function 
\begin{equation}
\label{cov}
\mathbb{E} \left[ b^{H} (t) b^{H} (s) \right] = R_{H}(t,s) := \frac{1}{2} ( t^{2H} + s^{2H}- \vert t-s \vert ) ^{2H}, \hskip0.5cm s,t \in [0,T].
\end{equation}
\quad The fractional Brownian motion ({\it fBm}) can also be defined as the only self-similar Gaussian process with stationary increments.

\quad Denote by ${\mathcal{H}}$ its associated canonical Hilbert space (reproducing kernel Hilbert space). If $H=\frac{1}{2}$ then $b^{\frac{1}{2}} = b$ is the
standard Brownian motion (Wiener process) and in this case ${\mathcal{H} } = L^{2}([0,T])$. Otherwise $\mathcal{H}$ is the  Hilbert space  on $[0,T]$ extending the set of indicator function $\mathbf{1}_{[0,T]}, t\in [0,T]$ by linearity and closure under the inner product
\[
\left\langle \mathbf{1}_{[0,t]} ; \mathbf{1}_{[0,s]}\right\rangle _{\mathcal{H}} = R_{H}\left( t , s \right) 
\]

\quad As the fBm is a Volterra process only for $H > 1/2$, we will focus our analysis exclusively in this case. In order to define the concept of mild-solution through convolution integrals, we need to recall the definition of  integrals with respect to the fBm. The followings facts will be needed in the sequel (we refer to \cite{N} or \cite{PT-2000} for their proofs):

$\bullet$ The fBm \index{fractional Brownian motion} admits a representation as Wiener integral \index{Wiener integral}of the form
\begin{equation} \label{BH}
b ^{H} (t) = \int_{0}^{t}K_{H} (t,s) db (s), 
\end{equation}
where $b = \{b (t), t\in [0, T] \}$ is a Wiener process, and $K_{ H}(t,s)$ is the kernel



\begin{equation}
 \label{K}
K_{H}(t,s)= c_{H} s^{\frac{1}{2}-H} \int _{s}^{t} (u-s)^{H-\frac{3}{2}} u^{H-\frac{1}{2}}  du
 \end{equation}
 where $t>s$ and $c_{H} =\left( \frac{ H(2H - 1) }{\beta( 2-2H, H-\frac{1}{2}) } \right)
^{\frac{1}{2}}.$

$\bullet$ For every $s < T$, consider the operator $K_{H}^{*} : \mathcal{H} \mapsto L^{2}([0, T])$, defined by 



\begin{equation} \label{Transfer_Op}
\displaystyle
	( K_{H}^{*} \, 	\phi )(s) = \int_{s}^{T} \phi(t) (s) \frac{ \partial K_{H} }{ \partial t } (t , s) \, dt.
\end{equation}

\quad Notice that, $ \left(K_{H}^{*} \phi \mathbf{1}_{[0, t]} \right) (s) = K_{H} (t , s) \phi (s) \mathbf{1}_{[0, t]} (s)  $, and the operator $K_{H}^{*}$ is an isometry between $\mathcal{H}$ and $L^{2} ([0, T])$ (see \cite{AMN-1999} or \cite{N}). Hence, for every $\phi \in \mathcal{H}$ it is possible to establish the following relationship between a Wiener integral with respect to the fBm and a Wiener integral with respect to the standard Brownian motion $b$ 
\begin{equation} \label{Int_rel}
\displaystyle
	\int_{0}^{t} \phi (s) db^{H}(s) = \int_{0}^{t} \left( K_{H}^{*} \phi \right) (s) db (s),
\end{equation}
for every $t \in [0 , T]$ and $\phi \mathbf{1}_{[0 , t]} \in \mathcal{H}$ if and only if $K_{H}^{*}  \phi \in L^{2}([0 , T])$.


\quad In general, the existence of the right-hand side of (\ref{Int_rel}) requires careful justification (see \cite{N}, Section 5.1). As we will work only with Wiener integrals over Hilbert spaces, we point out that if $X$ is a Hilbert space and $f \in L^{2}([0 , T] ; X)$ is a deterministic function, then relation (\ref{Int_rel}) holds, and the right hand-side is well defined in $L^{2}(\Omega ; X)$ if $K_{H}^{*} f$ is in $L^{2}([0 , T] \times X)$.

\subsection{Cylindrical Fractional Brownian motion} \label{2.2}

%
%

\quad As in \cite{Dun} or \cite{Tindel}, we define the standard {\it cylindrical} fractional Brownian motion in $X$ as the formal series
\begin{equation}\label{Cylin-fBm}
\displaystyle
	B^{H} (t) = \sum_{n = 0} ^{\infty} e_{n} b_{n}^{H} (t),
\end{equation}
where $\{ e_{n}, n \in \mathbb{N} \}$ is a complete orthonormal basis in $X$. It is well known that the infinite series (\ref{Cylin-fBm}) does not converge in $L^{2}(\mathbb{P})$, hence $B^{H} (t)$ is not a well-defined $X$-valued random variable. Nevertheless, for every Hilbert space $X_{1}$ such that $X \hookrightarrow X_{1}$, the linear embedding is a Hilbert-Schmidt operator, therefore, the series (\ref{Cylin-fBm}) defines a $X_{1}$-valued random variable and $\{ B^{H}(t), t \geq 0 \}$ is a $X_{1}$-valued $Id$-fBm. 

\quad Following the approach for a cylindrical Brownian motion introduced in \cite{Dapra}, it is possible to define a stochastic integral of the form
\begin{equation} \label{Integral-CfBm}
\displaystyle
	\int_{0}^{T} f(t) 	dB^{H}(t),	
\end{equation}
where $ f : [0, T] \mapsto \mathcal{L} (X, Y) $ and $Y$ is another real and separable Hilbert space, and the integral (\ref{Integral-CfBm}) is a $Y$-valued random variable that is independent of the choice of $X_{1}$.

\quad Let $f$ be a deterministic function with values in $\mathcal{L}_{2} (X, Y) $, the space of Hilbert-Schmidt operators from $X$ to $Y$. 
	We consider the following assumptions on $f$.
\begin{itemize} \label{Hyp-f}
	\item[i.-] For each $x \in X$, $f( \cdot )x \in L^{p}([0 , T] ; Y)$, for $ p > 1/H $.
	
	\item[ii.-] $ \alpha_{H} \int_{0}^{T} \int_{0}^{T} | f(s) |_{\mathcal{L}_{2} (X, Y)} \, | f(t) |_{\mathcal{L}_{2} (X, Y)} |s - t|^{2H-2} ds \, dt \, < \infty $.
\end{itemize}

\quad The stochastic integral (\ref{Integral-CfBm}) is defined as
\begin{equation} \label{Def_Integral_CfBm}
\displaystyle
	\int_{0}^{t} f(s) dB^{H}(s) := \sum_{n=1}^{\infty}    \int_{0}^{t}   f(s) e_{n}  db_{n}^{H}(s)
		 = \sum_{n=1}^{\infty} \int_{0}^{t} ( K_{H}^{\ast} f e_{n}) (s)  d b_{n}(s),
\end{equation}
where $b_{n}$ is the standard Brownian motion linked to the fBm $b_{n}^{H}$ via the representation formula (\ref{BH}). Since $f e_{n} \in L^{2}([0 , T] ; Y)$ for each $n \in \mathbb{N}$, the terms in the series (\ref{Def_Integral_CfBm}) are well defined. Besides, the sequence of random variables $\left\lbrace \int_{0}^{t}   f e_{n}  db_{n}^{H} \right\rbrace$ are mutually independent (see \cite{Dun}).

\quad The series (\ref{Def_Integral_CfBm}) is finite if
\begin{equation}
\displaystyle
	\sum_{n} \| K_{H}^{*}(f e_{n}) \|_{L^{2}([0, T]; V)}^{2} \, = \, \sum_{n} \| \, \| f e_{n} \|_{\mathcal{H}}  \, \|_{V}^{2} \, < \, \infty.	
\end{equation}

\quad If we consider $X = Y = \mathcal{H}$, we have
\begin{equation}
\begin{split}
	\sum_{n=1}^{\infty} \int_{0}^{t} f(s) e_{n} db_{n}^{H}(s) =& \sum_{n=1}^{\infty} \sum_{m=1}^{\infty} e_{m}  \int_{0}^{t}   \left\langle f(s)e_{n} , e_{m} \right\rangle_{\mathcal{H}}  d b_{n}^{H}(s) \\
  =& \sum_{n=1}^{\infty} \sum_{m=1}^{\infty} e_{m}  \int_{0}^{t}  \left\langle K_{H}^{\ast} (f(s) e_{n}) , e_{m} \right\rangle_{\mathcal{H}}  d b_{n}(s) \\
  =& \sum_{n=1}^{\infty} \int_{0}^{t} K_{H}^{\ast} (f(s) e_{n}) db_{n}(s) .
\end{split} 
\end{equation}

\subsection{Semigroup}    \label{2.3}

\quad It is well known that the Laplacian $\Delta$ is the infinitesimal generator of an analytic, strongly continuous semi-group of linear operators $( S(t), t \geq 0 )$ acting on $L^{p}(U)$ and given by $S(t) = e^{- t\Delta}$. Besides, for bounded domains the following estimate holds  (see \cite{Weis3})  

\begin{equation}\label{semi}
\| S(t) u \|_{p}\leq  \frac{1}{t^{\frac{d}{2}(1/r-1/p) } } \| u \|_{r} , \,   for \, 1< r\leq p < \infty.
\end{equation}


\section{Results} \label{Results}

\quad In this section we study the parabolic problem (\ref{para}) in the space $L^{p}(U)$. The required hypothesis are introduced as well as the notion of mild-solution.

\subsection{Hypothesis}

\quad We assume that $F$ is a nonlinear mapping from $L^{p}(U)$ onto $L^{m}(U)$ such that $F(0) = 0$, and for some $\alpha> 0$ and $m=\frac{p}{1+\alpha}$, it holds the estimate 

\begin{equation}\label{h1}
\|F(u)-F(v) \|_{ m } \leq  C \|u-v\|_{ p } (  \| u \|_{ p }^{\alpha} +  \| v \|_{ p }^{\alpha}), 
\end{equation}
with $C$ a positive constant. 

\quad In addition, the initial condition satisfies 

\begin{equation}\label{h2}
u_{0}\in L^{p}(U).
\end{equation}

\quad Besides, the cylindrical fBm $B^{H}$ has selfsimilarity parameter $H > 1/2$ and

\begin{equation}\label{h3}
H > d/4, \, \, p\cdot H \geq 1, \, \,  \mbox{and} \, \,  2p > \alpha d.
\end{equation}


\subsection{Mild-solution}

\quad Within the framework of paragraph \ref{2.2} we consider $X=L^{2}(U)$, $f = S(t - \cdot )$ and the complete orthonormal basis $\left\{e_{n}\right\}_{n \in \mathbb{N}}$
of eigenfunctions of the Laplacian operator, the stochastic convolution is given by

\[
\int_{0}^{t} S(t-s) dB^{H}(s) \, = \,  \sum_{j=1}^{\infty}   \int_{0}^{t}  S(t-s)e_{j}  d\beta_{j}^{H}(s). 
\]


Consider the mild formulation of equation (\ref{para}) (see \cite{Dun})

\begin{equation}\label{mild}
u(t)= S(t)u_{0}+  \ \int_{0}^{t} S(t-s) F ( u(s) ) \,  ds +  \int_{0}^{t} S(t-s) dB^{H}(s).
 \end{equation}

\begin{definition} \label{Def-3.1} 
	A measurable function $u: \Omega\times [0,T]\mapsto L^{p}(U)$ is a mild solution of the equation (\ref{para}) if 

\begin{enumerate}
\item $u$ satisfies the mild formulation  (\ref{mild}) with probability one.

\item $u \in C([0,T], L^{p}(U) )$. 

\end{enumerate}

\end{definition}

\begin{definition} \label{Def-3.2}
	Let $T_{0}$ be a stopping time. A measurable function $u : \Omega\times [0,T]\rightarrow L^{p}(U)$
is a local mild solution of (\ref{para}) in $C([0,T_{0}], L^{p}(U))$ with stopping time $T_{0}>0$, if it satisfies Definition \ref{Def-3.1} on $[0, T_{0}]$. It is the unique local mild solution with stopping time $T_{0}$, if two solutions are modifications of each other on $[0,T_{0}]$.  
\end{definition}


\subsection{Existence}

\quad Consider the linear problem 
\begin{equation}\label{Linear_Problem}
\left \{
\begin{aligned}
    & \partial_t z(t) = \Delta z(t) +  \partial_{t} B^{H}_{t}, \quad t \in [0 , T]  \\
    & z|_{t=0} =  0,     
\end{aligned}
\right .
\end{equation}
whose mild solution is given by
\begin{equation*} \label{Linear_Sol}
z(t)=   \int_{0}^{t} S(t-s) dB^{H}(s).
\end{equation*}

\quad Denote 

\begin{equation*}
\begin{split}
K_{0} & := \max \left\lbrace \| u_{0}\|_{p} \, , \sup_{ t \in [0,T]} \| \int_{0}^{t}  S(t-s)  dB^{H}(s) \|_{p}  \right\rbrace  \\
	& = \max \left\lbrace \| u_{0}\|_{p} \, , \sup_{t \in [0,T]} \| z(t) \|_{p} \right\rbrace,
\end{split}
\end{equation*}

\begin{equation*} 
\tilde{C}(t) = \left \{
\begin{aligned}
    &  C  \ \frac{t^{1- \frac{d \alpha}{2p}}}{1- \frac{d \alpha}{2p}}  \  ( 6 K_{0} )^{\alpha},  \quad \mbox{if} \quad  \alpha \geq \frac{\ln (3)}{\ln (2)} ,   \\
    &  C  \ \frac{t^{1- \frac{d \alpha}{2p}}}{1- \frac{d \alpha}{2p}}  \  ( 3 K_{0} )^{\alpha + 1} , \quad  \mbox{if} \quad \alpha < \frac{\ln (3)}{\ln (2)} ,
\end{aligned}
\right .
\end{equation*}

and define

\begin{equation}\label{Stop_time}
T_{0}= \left \{
\begin{aligned}
    &  T,  \quad \mbox{if} \quad  \tilde{C}(T) < 1,   \\
    &  \inf \left\{ 0\leq t\leq T : \tilde{C}(t)\geq 1 \right\}, \quad  \mbox{if} \quad   \tilde{C}(T) \geq  1.  
\end{aligned}
\right .
\end{equation}

\begin{theorem}\label{exi} Assume hypothesis (\ref{h1}), (\ref{h2}), (\ref{h3}). Then there exists a  local  mild  solution $u\in C ([0,T_{0}], L^{p}( U ) )$. 
\end{theorem}

\begin{proof} 

	Since $H>\frac{d}{4} $ and $ p H \geq 1 $ ,  the results in \cite{Co} allow us to conclude that the mild solution $z$ to the linear problem (\ref{Linear_Problem})
is in $C([0, T], L^{p}(U))$. Therefore, 

\[
\sup_{ t \in [0,T]} \| \int_{0}^{t}  S(t-s)  dB^{H} (s) \|_{p}<\infty. 
\]

\quad Now, in order to construct a contraction that will allow us to use a fix point argument, let's assume that $ \| u \|_{C([0,T_{0}], L^{p}(U))} := \sup_{t \in [0,T_{0}]} \| u(t) \|_{p} \leq  3 K_{0}$. Set 

\[
G[u](t) := S(t)u_{0} + \, \int_{0}^{t} S(t-s)F(u(s)) \, ds + z(t).  
\]

\quad We shall show that $ \sup_{t \in [0,T_{0}]} \| G[u](t) \|_{p}\leq 3K_{0}$. We have 

\[
 \| G[u](t) \|_{p} \leq 
 \|S(t)u_{0}\|_{p}  + \int_{0}^{t} \|S(t-s)F(u(s))\|_{p}  \ ds + \|z(t)\|_{p}.
\]

\quad As $( S(t) )_{t\geq 0}$ is a semigroup of contractions, for every $t \geq 0$  
	
\begin{equation}\label{one}
\|S(t)u_{0}\|_{p} \leq  \|u_{0}\|_{p},  
\end{equation}
	
and 
		
\begin{equation} \label{two}
\begin{split}
 \int_{0}^{t}   \|S(t-s)F(u(s))\|_{p} \, ds & \leq    \int_{0}^{t}  (t-s)^{- \frac{d \alpha }{2p}} \, \| F(u(s)) \|_{\frac{p}{\alpha + 1}}  \,  ds \\
	& \leq   C  \int_{0}^{t}    (t-s)^{- \frac{d \alpha}{2p}}  \  \| u(s) \|_{p}^{\alpha + 1} \     \ ds,
\end{split}
\end{equation}
			
where we used   (\ref{semi})  and hypothesis (\ref{h1}). 		
	
\quad From  (\ref{one}) and  (\ref{two})  we deduce that




\begin{equation*}
\begin{split}
	\| G[u](t) \|_{p}  & \leq  2  K_{0}  +  C \int_{0}^{t}  	\ (t-s)^{- \frac{d \alpha}{2 p}}  \| u(s) \|_{p}^{\alpha +  1} \, ds  \\
	& \leq 2  K_{0}  +  C \int_{0}^{t}  	\ (t-s)^{- \frac{d \alpha}{2 p}} \left( \sup_{s \in [0, T_{0}] } \| u(s) \|_{p} \right)^{\alpha +  1}
	 \, ds  \\
	 & \leq 2  K_{0}  +  C(3  K_{0})^{\alpha + 1} \frac{t^{1- \frac{d \alpha}{2 p}}}{1- \frac{d \alpha}{2p}}.
\end{split}	
\end{equation*}

Hence,

\begin{equation*}
\begin{split}
 \sup_{[0,T_{0}]} \| G[u](t) \|_{p}  & \leq   2  K_{0}  +  C(3  K_{0})^{\alpha + 1} \frac{T_{0}^{1- \frac{d \alpha}{2 p}}}{1- \frac{d \alpha}{2p}}  \\
	& = 3  K_{0} \left( \frac{2}{3} +  \   C  \frac{T_{0}^{1- \frac{d \alpha}{2 p}}}{1- \frac{d \alpha}{2p}} \ (3 K_{0})^{\alpha} \right) \,  \leq 3 K_{0},
\end{split}
\end{equation*} 

whenever 

\begin{equation}\label{Cond_1}
 C  \ \frac{T_{0}^{1- \frac{d \alpha}{2p}}}{1- \frac{d \alpha}{2p}}  \,  (3 K_{0}) ^{\alpha + 1} < 1. 
\end{equation}


\quad We shall show now that $G: X \mapsto X$ is a contraction, where $X:= \{ u \in C([0, T_{0}], L^{p}(U)) : \| u \|_{C([0,T_{0}], L^{p}(U))} \leq 3K_{0} \} $.  Let Fix  $u,v \in X$  then $t \in [0 , T_{0}]$, we have 

\begin{equation}
\begin{split}
 \| G[u](t) - G[v](t) \|_{p} & \leq 
 \int_{0}^{t} \|S(t-s) \left( F(u(t) ) - F(v(t)) \right) \|_{p} \, ds   \\
	& \leq  \int_{0}^{t} \,  (t-s)^{- \frac{d \alpha}{2p}}  \| F(u(t)) - F(v(t)) \|_\frac{p}{\alpha+1} \, ds \\
	& \leq C \, \int_{0}^{t} \,  (t-s)^{- \frac{d \alpha}{2p}}  \| u(t) - v(t) \|_{p}  \,  \left( \| u(t) \|_{p} ^{\alpha} + \| v(t) \|_{p}^{\alpha}  \right) \, ds \\
	& \leq  C  (6 K_{0})^{\alpha} \,  \frac{T_{0}^{1- \frac{d \alpha}{2p}}}{1- \frac{d \alpha}{2p}}  \,   \sup_{t \in [0,T_{0}]}\|u(t) - v(t) \|_{p} \, ds,
\end{split}
\end{equation}

where we used   (\ref{semi})  and hypothesis (\ref{h1}). Hence, if
\begin{equation}\label{Cond_2}
 C  \ \frac{T_{0}^{1- \frac{d \alpha}{2p}}}{1- \frac{d \alpha}{2p}}  \,  (6 K_{0})^{\alpha} < 1,  
\end{equation}

then

\[
 \sup_{t \in [0,T_{0}]} \| G[u](t) - G[v](t) \|_{p}  \,  <  \, \sup_{t \in [0,T_{0}]} \| u(t) - v(t) \|_{p}. 
\]

\quad Therefore, $G$ is a contraction. Hence, there exist a unique fixed point. 



\end{proof}

\subsection{Example of a non-linearity F}

\quad An example of a non-linearity F satisfying condition (\ref{h1}) is as follows. Let $f$ be a mapping from $\mathbb{R}^{d}$ to $\mathbb{R}^{d}$ verifying $f(0) = 0 $ and
\[
|f(y) - f(x)|\leq C |x-y|(|x|^{\alpha}+ |y|^{\alpha}),
\]
for $\alpha > 0$.

\quad Set $F(u)(x) = f(u(x))$, hence, by H\"older's inequality $F$ satisfies (\ref{h1}). As an especific example to construct the non-linearity $F$, we may consider the function $f(x)=x|x|^{\alpha}$.

\section{Appendix} \label{Appendix}

\subsection{Preliminary}

\quad We follow the presentation in \cite{Co}.  Let $X$ be a real separable Hilbert space and $E$ be a Banach space. A bounded operator 
$R\in \mathcal{L}(X,E)$ is $\gamma$-radonifying provided that there exists a centered Gaussian probability 
$\nu$ on $E$ such that 
\[
\int_{E} \varphi(x) \ d\nu(x)= \| R^{\ast}\varphi\|_{X} ,\; \varphi\in E^{\ast}.
\]

Such a measure is at most one; therefore we set 
\[
\| R \|_{\gamma(X,E)}^{2}:= \int_{E}  \| x\|_{E}^{2} \ d\nu (x)
\]
and denote by $\gamma(X,E)$ the space of the $\gamma$-radonifying operator. It is well-know 
that $\gamma(X,E)$ equipped with $ \|  .\|_{\gamma(X,E)}^{2}$ is a separable Banach space
(see \cite{On}).

Let $(D,\mu)$ be a measure space. $L^{p}=L^{p}(D, \mu)$
is a separable Banach space for $1\leq p< \infty$.  Denote  $B^{H}$ a $X$-cylindrical 
fractional Brownian motion. If we consider an operator 
$G \in\gamma(X,L^{p}(D; \mathcal{H}))$, the stochastic integral can be defined  as 

\[
\int_{0}^{T} G(r)dB^{H}(r) = \sum_{n} \int_{0}^{T} Ge_{n} db_{n} ^{H} (r), 
\]

for more details see  \cite{Co}.

\subsection{Semi-linear Equation}

We consider the following stochastic differential equation 

\begin{equation}\label{paraB}
 \left \{
\begin{aligned}
    & \partial_t u( t ) = A u( t ) + F( u(t) ) + \Phi \partial_{t} B^{H}_{t}, \quad t \in [0 , T] \\
    & u|_{t=0}=  u_{0}, 
\end{aligned}
\right .
\end{equation}

where $u_{0}\in L^{p}$, $A : Dom(A) \subset L^{p} \mapsto L^{p}$, is the infinitesimal generetor of an analytic  strongly continuous semigroup of linear operators $(S(t), t\geq0)$
acting on $L^{p}$, and $ \Phi\in \gamma(X,L^{p})$ .

\quad Consider the mild formulation of equation (\ref{paraB})

\begin{equation}\label{mild2}
u(t)= S(t)u_{0}+  \ \int_{0}^{t} S(t-s) F ( u(s) ) \,  ds +  \int_{0}^{t} S(t-s)\Phi dB^{H}(s).
 \end{equation}

\begin{definition} \label{Def-3.1B} 
	A measurable function $u: \Omega\times [0,T]\mapsto L^{p}$ is a mild solution of the equation (\ref{paraB}) if 

\begin{enumerate}
\item $u$ satisfies the mild formulation  (\ref{mild2}) with probability one.

\item $u \in C([0,T], L^{p} )$. 

\end{enumerate}

\end{definition}

\subsection{Hypothesis}

\quad We assume that $F$ is a nonlinear mapping from $L^{p}$ onto $L^{m}$ such that $F(0) = 0$, and for some $\alpha> 0$ and $m=\frac{p}{1+\alpha}$, it holds the estimate 

\begin{equation}\label{h1B}
\|F(u)-F(v) \|_{ m }\leq  C \|u-v\|_{ p } (  \|u\|_{ p }^{\alpha} +  \|v\|_{ p }^{\alpha}), 
\end{equation}
with $C$ a positive constant. 

\quad Besides, we assume that  $2p > \alpha d$, 

\begin{equation}\label{h2B}
u_{0}\in L^{p},
\end{equation}

and the Hurst parameter $H$ is bigger than $1/2$ and satisfies

\begin{equation}\label{h3B}
H > d/4, \, \, p\cdot H \geq 1.
\end{equation}

In addition, we assume that for $ 1< r\leq p < \infty$, $\lambda \in [0, H)$, for all  $u\in L^{p}$  and $S(t)\Phi\in \gamma(X,L^{p})$,

\begin{equation}\label{semiB}
\| S(t) u \|_{p}  \leq  t^{- \frac{d}{2}(1/r-1/p) }  \| u \|_{r} , \,  
\end{equation}

and

\begin{equation}\label{semiB2}
\|S(t)\Phi \|_{\gamma(X,L^{p})}\leq  t^{-\lambda},
\end{equation}

\begin{remark} According to Corollary 4.3 in \cite{Co} we deduce that 

\[
\sup_{ t \in [0,T]}\|\int_{0}^{T}  S(t-s)\Phi  dB^{H}\|_{p}<\infty. 
\]

\end{remark}

\subsection{Result}

\begin{theorem} Assume hypothesis (\ref{h1B}), (\ref{h2B}), (\ref{h3B}), 
(\ref{semiB}), (\ref{semiB2}). Then there exists  a unique   local  mild  solution $u\in C ([0,T_{0}], L^{p})$. 
\end{theorem}

\begin{proof}

The proof follow the same steps that the proof of theorem \ref{exi}. 

\end{proof}

\section*{Acknowledgements}

Christian Olivera  is partially supported by FAPESP 		by the grants 2017/17670-0 and 2015/07278-0.
Also supported by CNPq by the grant
		426747/2018-6.


\end{document}